\documentclass[a4paper,11pt]{amsart}

\usepackage{amsmath,amssymb,amsthm}

\numberwithin{equation}{section} \theoremstyle{plain}
\newtheorem{thm}{Theorem}[section]

\newtheorem{prop}[thm]{Proposition}
\newtheorem{lem}[thm]{Lemma}
\newtheorem{cor}[thm]{Corollary}

\newtheorem{conj}[thm]{Conjecture}

\newtheorem{rem}[thm]{Remark}

\newtheorem*{acknow}{Acknowledgments}
   
 \makeatletter

\def\<{\langle}
\def\>{\rangle}
\def\({\left(}
\def\){\right)}
\def\[{\left[}
\def\]{\right]}
\def\tr{\mathop{\text{tr}}}

\makeatother

\title[A weak version of the Perdomo Conjecture]{A lower bound for $L_2$ length of second fundamental form on minimal hypersurfaces}

\author[J. Q. Ge]{Jianquan Ge}
\address{School of Mathematical Sciences, Laboratory of Mathematics and Complex Systems, Beijing Normal University, Beijing 100875, P.R. CHINA.}
\email{jqge@bnu.edu.cn}

\author[F. G. Li]{Fagui Li$^{*}$}
\address{School of Mathematical Sciences, Laboratory of Mathematics and Complex Systems, Beijing Normal University, Beijing 100875, P.R. CHINA.}
\email{faguili@mail.bnu.edu.cn}

\date{}
\thanks {$^{*}$ the corresponding author.}
\thanks{The first author is partially supported by Beijing Natural Science Foundation (No. Z190003).}

\begin{document}
\maketitle

\begin{abstract}
We prove a weak version of the Perdomo Conjecture, namely, there is a positive constant
  $\delta(n)>0$ depending only on $n$ such that on any closed embedded, non-totally geodesic, minimal hypersurface $M^n$ in $\mathbb{S}^{n+1}$,
 $$\int_{M}S \geq \delta(n){\rm Vol}(M^n),$$
 where $S$ is the squared length of the second fundamental form of $M^n$. The Perdomo Conjecture asserts that $\delta(n)=n$ which is still open in general.
 As byproducts, we also obtain some integral inequalities and Simons-type pinching results on  closed embedded  (or immersed) minimal hypersurfaces, with the first positive eigenvalue $\lambda_1(M)$ of the Laplacian involved.
\end{abstract}

\section{Introduction}
Half a century ago, S. S. Chern \cite{Chern68}  proposed the following famous conjecture.
\begin{conj} \label{Conjecture SSChern 1965}
Let $M^n$ be a  closed minimal  hypersurface of constant scalar curvature \emph{(CSC)} $R_M$ immersed in the unit sphere $\mathbb{S}^{n+1}$. Then the set of all possible values of $R_M$  is discrete.
\end{conj}
S. T. Yau raised it again as the 105th problem in his Problem Section \cite{Yau  1982}.  The refined version of the Chern Conjecture can be stated as follows \cite{M Scherfner S Weiss and  Yau 2012}.
\begin{conj} [\textbf{Chern Conjecture}]
\label{Conjecture Refined version of Chern Conjecture}
Let $M^n$ be a  closed  immersed minimal CSC hypersurface of $\mathbb{S}^{n+1}$. Then $M^n$ is isoparametric.
\end{conj}
The classification of isoparametric hypersurfaces in unit spheres was initiated in late 1930s by Cartan   
and finally completed till to the year 2020 by many mathematicians (cf. Cecil-Chi-Jenson \cite{CCJ07}, Immervoll \cite{Immervoll 2008}, Chi \cite{Chi11,Chi13,Chi16}, Dorfmeister-Neher \cite{Dorfmeister and  Neher 1985} and Miyaoka \cite{Miy13,Miy16}, etc.), please see the excellent book \cite{CR15} and the elegant survey \cite{Chi19} for more details and references.

 In 1968, J. Simons \cite{Simons68}  gave the first pinching result which motivated the Chern Conjecture, since the discreteness of $R_M$ is equivalent to that of $S:=|A|^2$ ($A$ is the shape operator) on minimal hypersurfaces by the Gauss equation $R_M=n(n-1)-S$.
\begin{thm}[\textbf{Simons inequality}]\label{thm J. Simons}
Let $M^n$ be a  closed immersed minimal hypersurface of $\mathbb{S}^{n+1}$  with squared length of the second fundamental form $S$.
Then
$$
\int_{M}S\left( S-n\right)\geq0.
$$
In particular, if $0\leq S\leq n$, one has either  $S\equiv0$ or $S\equiv n$  on $M^n$.
\end{thm}
 The classification
  of  $S\equiv n$  in Theorem \ref{thm J. Simons}  was characterized by Chern-do Carmo-Kobayashi \cite{Chern do Carmo Kobayashi 1970} and Lawson \cite{Lawson 1969}  independently. Namely,  \emph{the Clifford tori are the only closed minimal CSC hypersurfaces in $\mathbb{S}^{n+1}$  with
  $S \equiv n$}, i.e.,
  $$
  M^n=S^{k}(\sqrt{\frac{k}{n}})\times S^{n-k}(\sqrt{\frac{n-k}{n}}), \ \  1\leq k\leq n-1.
  $$

 In 1983, Peng and Terng \cite{Peng and Terng 1983,Peng and Terng1 1983} made the first breakthrough towards Chern  Conjecture \ref{Conjecture SSChern 1965}, namely, for minimal CSC hypersurfaces in $\mathbb{S}^{n+1}$,
\emph{if $S > n$, then $S > n+ \frac{1}{12n}$. Moreover, if $S > 3$ for $n = 3$, then $S \geq  6$.}  In 1993, Chang \cite{Chang 1993}  finished the proof of Chern  Conjecture \ref{Conjecture Refined version of Chern Conjecture} for $n=3$.
  Yang-Cheng \cite{Cheng Qing Ming and H Yang 1998}  and
  Suh-Yang  \cite{Suh Yang2007}  improved the second gap from $\frac{1}{12n}$ to $\frac{3n}{7}$. However, it is still an open problem for higher dimensional case whether $S \geq 2n$ if $S > n$. As for minimal isoparametric hypersurfaces in $\mathbb{S}^{n+1}$ with $g$ distinct principal curvatures, it indeed satisfies $S=(g-1)n\geq2n$ when we exclude the $g=1$ case (the equators) and the $g=2$ case (the Clifford tori).

Strongly supporting Chern Conjecture \ref{Conjecture Refined version of Chern Conjecture}, a recent remarkable progress by Tang-Wei-Yan  \cite{TWY18} and  Tang-Yan \cite{TY20} generalized the theorem of de Almeida and Brito \cite{Almeida and Brito 1990} for $n = 3$ to arbitrary dimension $n$. Namely, \emph{a closed immersed hypersurface $M^n$ in  $\mathbb{S}^{n+1}$ having constant $1,2,\cdots,(n-1)$-th mean curvatures and nonnegative scalar curvature $R_M\geq0$ is isoparametric}. de Almeida-Brito-Scherfner-Weiss \cite{Almeida  Brito Scherfner and  Weiss 2018} showed that
 \emph{a closed immersed hypersurface $M^n$ in  $\mathbb{S}^{n+1}$ having constant Gauss-Kronecker curvature $K_M$ and $3$ distinct principal curvatures everywhere is isoparametric.}
For the four dimensional case, Deng-Gu-Wei \cite{Deng Gu and  Wei  2017} proved that \emph{if $M^4$ is a closed Willmore minimal CSC hypersurface in $\mathbb{S}^5$, then it is isoparametric.} In other words, in dimension four \cite{Deng Gu and  Wei  2017}  dropped the nonnegativity assumption  $R_M\geq0$ of \cite{TY20}, under the new condition of being Willmore which is equivalent to that the third mean curvature vanishes other than being only a constant as in \cite{TY20}.
In fact, given some pinching restrictions other than identities to the third mean curvature and the Gauss-Kronecker curvature, one can also remove the nonnegativity assumption  $R_M\geq0$ of \cite{TY20} in dimension four (cf. \cite{Li21}).
More related results can be found in the  surveys  by Scherfner-Weiss \cite{M Scherfner S Weiss 2008}, Scherfner-Weiss-Yau \cite{M Scherfner S Weiss and  Yau 2012} and Ge-Tang \cite{Ge Tang 2012}. Very recently, we \cite{Ge Li 2020} gave a characterization to the condition of Chern Conjecture \ref{Conjecture Refined version of Chern Conjecture}, a Takahashi-type theorem, i.e., \emph{An immersed hypersurface $M^n$ in  $\mathbb{S}^{n+1}$  is minimal and has constant scalar curvature if and only if $\Delta \nu=\lambda\nu$ for some constant $\lambda$, where $\nu$ is a unit normal vector field of $M^n$}.

Without assuming constant scalar curvature,
 Peng and Terng \cite{Peng and Terng 1983,Peng and Terng1 1983} obtained 
 that 
  \emph{there exists a positive constant
   $\delta(n)$ depending only on $n$,
   such that if $n\leq S\leq n+\delta(n)$, $n\leq 5$,
   then $S\equiv n$, i.e., $M^n$ is a minimal Clifford torus.}
   Later, Cheng-Ishikawa \cite{Cheng Qing Ming and Ishikawa1999} improved the previous pinching constant for $n\leq 5$,  Wei-Xu \cite{Si Ming Wei  Hong Wei Xu2007} extended it to $n = 6, 7$, and  Zhang \cite{Qin Zhang 2010}  extended it to $n\leq8$.
Finally,   Ding and Xin \cite{Qi Ding  Y.L. Xin. 2011} proved the second gap for all dimensions, in particular,  they showed $\delta(n)=\frac{n}{23}$ for
 $n\geq6$.  Xu-Xu \cite{Xu H. Xu 2017} improved this pinching constant to $\delta(n)=\frac{n}{22}$ and  Li-Xu-Xu \cite{Li Lei Hongwei Xu  Zhiyuan Xu 2017} further improved it to $\delta(n)=\frac{n}{18}$. As a matter of fact in the pinching results above, the condition $S\geq n$ is indispensable owing to  some counterexamples of Otsuki \cite{Otsuki 1970}.

In this paper, without assuming constant scalar curvature, we are interested in whether there is a universal lower bound for the mean value of $S$ on non-totally geodesic minimal hypersurfaces.
Inspired by the preceding paper \cite{Ge Li 2020}, we answer this question affirmatively for embedded hypersurfaces.
 \begin{thm}[\textbf{Main Theorem}]\label{Corollary  gap  results of Perdomo Conjecture in the introduction}
 Let $M^n$ be a closed embedded, non-totally geodesic, minimal hypersurface in $\mathbb{S}^{n+1}$.
 Then there is  a positive constant
  $\delta(n)>0$, depending only on $n$,
  such that
 $$\int_{M}S \geq \delta(n){\rm Vol}(M^n).$$
 \end{thm}

In fact, Theorem \ref{Corollary  gap  results of Perdomo Conjecture in the introduction} provides an evidence to the following Perdomo Conjecture.
\begin{conj}[\textbf{Perdomo  Conjecture \cite{Perdomo 2004 average of the scalar curvature}}]\label{Conjecture Perdomo 2004}
Let $M^n$ be a  closed  embedded, non-totally geodesic, minimal hypersurface in $\mathbb{S}^{n+1}$, then
$$\int_{M}S \geq n{\rm Vol}(M^n).$$ Moreover, the equality holds if and only if $S\equiv n$, i.e., $M^n$ is a minimal Clifford torus.
\end{conj}

The equality assertion of Perdomo Conjecture \ref{Conjecture Perdomo 2004} for surfaces is equivalent to the Lawson
 Conjecture \cite{Lawson 1970}, i.e., \emph{The only embedded minimal torus in $\mathbb{S}^3$ is the Clifford torus.}
 This is because by the Gauss equation and the Gauss-Bonnet theorem, for genus $g$ minimal surface $M^2\subset\mathbb{S}^3$, we have
 $$\int_{M}S=8\pi\left( g-1\right) +2{\rm Vol}(M^2).$$
 Notice that the inequality also holds for genus $g\geq2$ surfaces, while for the $g=0$ case, it follows from Almgren \cite{Alm66} and Calabi \cite{Cal67} that any embedded minimal sphere in $\mathbb{S}^3$ is totally geodesic (which is not true for higher dimensions by Hsiang \cite{Hsi83}).
 The Lawson Conjecture has been proven by Brendle \cite{Brendle S 2013,Brendle S 2013 survey of recent results}.

For general dimension, Perdomo \cite{Perdomo 2004 average of the scalar curvature} also conjectured that  \emph{the only minimal immersed hypersurfaces in $\mathbb{S}^{n+1}$ with ${\rm Index}(M^n) = n + 3$ are the minimal Clifford tori.}
With an additional assumption on the symmetries of $M^n$, this conjecture was verified by Perdomo himself \cite{Perdomo 2001 Low index}.
In particular, if  Conjecture \ref{Conjecture Perdomo 2004} is true, then the conjecture above also holds for embedded hypersurfaces \cite{Perdomo 2004 average of the scalar curvature}.
In addition, 
Perdomo \cite{Perdomo 2004 Rigidity of minimal hypersurfaces} proved that the  condition  ``$M^n$ is embedded''  is needed in Conjecture \ref{Conjecture Perdomo 2004}. This is due to the fact that rotational minimal hypersurfaces (those with exactly two principal curvatures) satisfy the reverse inequality and it is known that none of these examples are embedded.
 For more relations between the Jacobi (or stability) operator and Perdomo's conjectures we refer to \cite{Alias Brasil and Barros 2005,Alias Brasil and Perdomo 2007,Perdomo 2001 Low index, Perdomo 2004 average of the scalar curvature}.
 Besides, assume there are $(n + 2)$ great
 hyperspheres of $\mathbb{S}^{n+1}$ perpendicular to each other, such that $M^n$ is symmetric with respect to them, then  Conjecture \ref{Conjecture Perdomo 2004} was verified by Wang and Wang \cite{C P Wang and P Wang 2020} very recently.
  There are some similar pinching results (cf. \cite{Gu J R  Xu H W Xu Z Y 2016,Jun-Min L Chang 1989,C. L. Shen 1989,Xu  Hong-Wei2006}).

 In fact, Theorem \ref{Corollary  gap  results of Perdomo Conjecture in the introduction} is a corollary of the general inequality below.
  \begin{thm}\label{thm  gap  results of Perdomo Conjecture in the introduction}
   Let $M^n$  be a  closed  embedded hypersurface in $\mathbb{S}^{n+1}$ and $\lambda_1(M)$ be the first positive eigenvalue of the Laplacian. Then
  $$
  \int_{M}S
  \geq
  \lambda_1(M)
  \frac{{\rm Vol}^2(M^n)-{\rm Vol}^2(\mathbb{S}^{n})}{{\rm Vol}(M^n)},
  $$
  where the equality holds if and only if $M^n$ is totally geodesic.
   \end{thm}
 The following bound of the first eigenvalue (which can be compared with the Yang-Yau inequality $\lambda_1(M) \leq \frac{8\pi(1+g)}{{\rm Vol}(M^2)}$ for genus $g$ (minimal) surfaces, see \cite{Schoen Richard and S. T. Yau 1994}), and the Simons-type pinching result for general (not only minimal) closed hypersurfaces are immediate corollaries of Theorem \ref{thm  gap  results of Perdomo Conjecture in the introduction}.
   \begin{cor}\label{Corollary  lambda1M  results of embedded  hypersurface in the introduction}
 Let $M^n$  be a  closed  embedded hypersurface in $\mathbb{S}^{n+1}$. If
 ${\rm Vol}(M^n)> {\rm Vol}(\mathbb{S}^{n})$,
 then
 $$
    \lambda_1(M)<
   \frac{{\rm Vol}(M^n)\int_{M}S}{{\rm Vol}^2(M^n)-{\rm Vol}^2(\mathbb{S}^{n})}.
   $$
 \end{cor}

 \begin{cor}\label{Corollary  gap  results of embedded (totally geodesic) hypersurface in the introduction}
 Let $M^n$  be a  closed  embedded hypersurface in $\mathbb{S}^{n+1}$. If
  $$
 \int_{M}S
  \leq
  \lambda_1(M)
  \frac{{\rm Vol}^2(M^n)-{\rm Vol}^2(\mathbb{S}^{n})}{{\rm Vol}(M^n)},
  $$
 then $ M^n$ is totally geodesic.
 \end{cor}
  Cheng-Li-Yau \cite{Cheng Li Yau 1984 Heat equations} proved in 1984 that \emph{if $M^n$ is a closed immersed minimal hypersurface in  $\mathbb{S}^{n+1}$ and $M^n$ is non-totally geodesic, then there is  a positive constant $c(n)>0$, depending only on $n$, such that the area  of  $M^n$ satisfies
  $$
  {\rm Vol}(M^n)>\left( 1+c(n)\right) {\rm Vol}(\mathbb{S}^n).
  $$}
  Thus, Theorem \ref{Corollary  gap  results of Perdomo Conjecture in the introduction} follows from Theorem \ref{thm  gap  results of Perdomo Conjecture in the introduction} as $\lambda_1(M)\geq\frac{n}{2}$ by Choi and Wang \cite{Choi Wang 1983}.

Throughout this paper, we denote by $S$ the squared length of the second fundamental form and its maximum and minimum by
  $$
  S_{\max}=\sup_{p\in M^n}S(p),\ \
  S_{\min}=\inf_{p\in M^n}S(p).
  $$
For immersed case we also obtain similar inequalities and pinching results as follows.
 \begin{thm}\label{thm  some results of Perdomo Conjecture in the introduction}
 Let $M^n$  be a  closed  immersed
  minimal  hypersurface in $\mathbb{S}^{n+1}$. 
\begin{itemize}
\item [(i)]  If $S\not\equiv 0$, then
$$ \int_{M}S\geq
 \left( \lambda_1(M)-\frac{2(n-1)}{\lambda_1(M)(2n-1)}\left(S_{\max}-n \right) \left( S_{\max}-S_{\min}\right) \right){\rm{Vol}}(M^n).$$
\item [(ii)] The following Simons-type inequality holds:
$$
\begin{aligned}
\int_{M}S^2
&
\geq
\frac{n}{n-1}\int_{M}S
\left( \frac{2n-1}{n}\lambda_1(M)-\frac{\int_{M}S}{{\rm{Vol}}(M^n)  }\right).
\end{aligned}$$
If $S\not\equiv 0$, then
$$
n\int_{M}S+(n-1)S_{\max}{\rm{Vol}}(M^n)\geq (2n-1)\lambda_1(M){\rm{Vol}}(M^n).
$$
\end{itemize}
\end{thm}
\begin{cor}\label{Corollary the week result of Perdomo Conjecture}
Let $ M^n$ be a closed   immersed  minimal hypersurface  in  $\mathbb{S}^{n+1}$. If
\begin{itemize}
\item[(i)]  $
\int_{M}S
\leq
 \frac{(2n-1)\lambda_1^2(M)}{(2n-1)\lambda_1(M)+2n(n-1)}{\rm{Vol}}(M^n);$
\item[(ii)]  $S_{\max}-S_{\min}< n;$
\end{itemize}
then   $M^n$ is totally geodesic.
\end{cor}

\begin{rem}
If $M^n$ is embedded, then  by Choi and Wang \cite{Choi Wang 1983} we can replace $\lambda_1(M)$ with $n/2$
   in Theorems $\ref{thm  gap  results of Perdomo Conjecture in the introduction}$ and $\ref{thm  some results of Perdomo Conjecture in the introduction}$, Corollaries $\ref{Corollary  gap  results of embedded (totally geodesic) hypersurface in the introduction}$ and $\ref{Corollary the week result of Perdomo Conjecture}$. These results can be improved further by replacing $\lambda_1(M)$ with $n$ if the following Yau Conjecture is true.
\end{rem}

 \begin{conj}[\textbf{Yau Conjecture \cite{Yau  1982,Yau  1992}}] \label{conj  Yau 1982}
~~
\begin{itemize}
\item[(i)] Let $ M^n$ be a closed   embedded minimal hypersurface of $\mathbb{S}^{n+1}$, then
 $\lambda_1 (M)=n$.
 \item[(ii)] The area of one of the minimal Clifford tori  gives the lowest value of area among all non-totally geodesic closed minimal hypersurfaces of $\mathbb{S}^{n+1}$.
  \end{itemize}
 \end{conj}
 In particular, Tang and Yan \cite{TY13} proved Yau Conjecture \ref{conj  Yau 1982} (i) in the isoparametric case.
There are  some  classical results on the first  eigenvalue for minimal hypersurfaces in spheres (cf. \cite{Brendle S 2013 survey of recent results,J Choe 2006,J Choe and M Soret  2009,Choi Wang 1983,TXY14}, etc.).
For (ii) of Yau Conjecture \ref{conj  Yau 1982} (also called the Solomon-Yau Conjecture \cite{Ge19}), among minimal rotational hypersurfaces Perdomo and Wei \cite{Perdomo and  Wei 2015}  showed numerical evidences that it is true if $2\leq n\leq 100$ and
 Cheng-Wei-Zeng
\cite{Cheng Qing Ming Guoxin Wei and Yuting Zeng 2019}  showed in all dimensions. Remarkably, in the asymptotic sense, Ilmanen-White \cite{IW15} verified the Solomon-Yau Conjecture in the class of topologically nontrivial hypercones.

\section {Preliminary lemmas and a Simons-type pinching result}
In this section, we will give some necessary lemmas 
which also lead to a Simons-type pinching result (see Proposition \ref{prop A piniching  Inequality of minimal hypersurfaces in the spheres}).

Let $x: M^n\looparrowright    \mathbb{S}^{n+1} \subset \mathbb{R}^{n+2}$ be a closed immersed hypersurface  in the unit sphere  $\mathbb{S}^{n+1}$ and  $\nu(x)$ denote the unit normal vector field,
$\nabla$ and $\overline\nabla$  be the Levi-Civita connections on $M^n$ and  $\mathbb{S}^{n+1}$,  respectively. Following \cite{Ge Li 2020}, for any unit vector $a \in \mathbb{S}^{n+1}$, we consider the height functions on $M^n$,
$$\varphi_a(x) = \langle x,a \rangle,\quad  \psi_a (x)=\langle \nu,a \rangle.$$
Then we have the following basic properties.
\begin{prop} \cite{Ge Li 2020, Nomizu and Smyth 1969} \label{prop funda}
For all $a\in \mathbb{S}^{n+1}$, we have
$$\begin{array}{lll}
\nabla \varphi_a=a^{\rm T},&
\nabla \psi_a =-Aa^{\rm T},\\
\Delta \varphi_a=-n\varphi_a+nH\psi_a ,&
\Delta \psi_a =-n\left\langle \nabla H, a \right\rangle +nH\varphi_a -S\psi_a. \\
\end{array}$$
 where  $a^{\rm T}\in \Gamma(TM)$ denotes the tangent component of $a$ along $M^n$, $A$ is the shape operator with respect to $\nu$, i.e., $A(X)=-\overline\nabla_X\nu$, $S=\|A\|^2=\tr AA^t$, $H = \frac{1}{n} \tr A$ is the mean curvature.
\end{prop}

\begin{lem} {\rm\textbf{$($Choi and Wang \label{lem  the first Dirichlet eigenvalue of the Laplacian Choi Wang 1983} \cite{Choi Wang 1983}$)$}}
Let $ M^n$ be a closed embedded minimal hypersurface in $\mathbb{S}^{n+1}$, then
$\lambda_1 (M)\geq n/2$. 
\end{lem}

A careful argument (see \cite[ Theorem 5.1]{Brendle S 2013 survey of recent results}) shows that the strict inequality holds, i.e., $\lambda_1 (M)> n/2$ in Lemma \ref{lem  the first Dirichlet eigenvalue of the Laplacian Choi Wang 1983}.

\begin{lem}\label{lem A similar Simons Inequality of minimal hypersurfaces in the spheres}
Let $ M^n$ be a closed   immersed minimal hypersurface  in  $\mathbb{S}^{n+1}$, then
\begin{equation*}
\begin{aligned}
\int_{M}S\left(S-n\right)
(S-S_{\min})
&\geq
\frac{\lambda_1(M)}{2}\left( \int_{M}S^2-
\frac{\left( \int_{M}S\right) ^2}
{{\rm Vol}(M^n)}\right).
\end{aligned}
\end{equation*}
Moreover, if $M^n$ is  embedded, 
then $$
\int_{M}S\left(S-n\right)
(S-S_{\min})
\geq
\frac{n}{4}\left( \int_{M}S^2-
\frac{\left( \int_{M}S\right) ^2}
{{\rm Vol}(M^n)}\right).
$$
\end{lem}

\begin{proof}
Without loss of generality,
we suppose  $S$ is not a constant on $M^n$.
Let  $\mathcal{F}$ be  the set of non-constant   functions
$f: M^n\to \mathbb{R}$ with $\int_{M}f=0$ and $f\in H^1(M)$. Recall 
 the Simons identity (cf. \cite{Simons68})
\begin{equation}\label{equation Simons equation}
\frac{1}{2}\Delta S=|\nabla A|^2+S(n-S).
\end{equation}
Let $u=S-\frac{\int_{M}S}{{\rm Vol}(M)}\in\mathcal{F}$, one has
$$
\begin{aligned}
\lambda_1( M)
&=\inf_{\substack {f\in \mathcal{F}}}
 \frac{\int_{M}
|\nabla f|^2 }{\int_{M} f^2 }
\leq \frac{\int_{M}
|\nabla u|^2 }{\int_{M} u^2 }\\
&=
 \frac{\int_{M}
|\nabla S|^2 }{\int_{M} u^2 }
= \frac{-\int_{M}
S\Delta S }{\int_{M} u^2 }\\
&=\frac{-2\int_{M}
S\left(  |\nabla A|^2+S(n-S)  \right)  }{\int_{M} u^2}\\
&\leq\frac{-2\int_{M}
\left( S_{\min} |\nabla A|^2+S^2(n-S) \right)  }{\int_{M} u^2 }\\
&=\frac{2\int_{M}S\left(S-n\right)
(S-S_{\min})
}
{\int_{M} u^2 },
\end{aligned}
$$
and 
\begin{equation}\label{equation cauchy inequ}
\int_{M} u^2=\int_{M}S^2-
\frac{\left( \int_{M}S\right) ^2}
{{\rm Vol}(M^n)}
\geq 0.
\end{equation}
This proves the first inequality and the second follows from Lemma \ref{lem  the first Dirichlet eigenvalue of the Laplacian Choi Wang 1983}.
\end{proof}

From  Lemma \ref{lem A similar Simons Inequality of minimal hypersurfaces in the spheres}, we have the following pinching result immediately.
\begin{prop}\label{prop A piniching  Inequality of minimal hypersurfaces in the spheres}
Let $ M^n$ be a closed  embedded  minimal hypersurface  in  $\mathbb{S}^{n+1}$. If
\begin{itemize}
\item[(i)]  ${\int_{M} S }\leq n {\rm Vol}(M^n)$;
\item[(ii)]  ${\int_{M} S^2\left( S-n\right)  }\leq
S_{\max} {\int_{M} S\left( S-n\right)  }$;
\item[(iii)]  $S_{\max}-S_{\min}< \frac{n}{4}$;
\end{itemize}
then   $S\equiv n$ or $S\equiv 0$, i.e., $M^n$ is a minimal Clifford torus or an equator.  Moreover,  if Yau  Conjecture $\ref{conj  Yau 1982}$ $(i)$ is true,
condition $(iii)$ can be replaced by
$$S_{\max}-S_{\min}< \frac{n}{2}.$$
\end{prop}
\begin{proof}
By the Simons identity (\ref{equation Simons equation}) and condition $(i)$, we have
$$
\begin{aligned}
\int_{M}S^2-
\frac{\left( \int_{M}S\right) ^2}
{{\rm Vol}(M^n)}
&=
n\int_{M}S-\frac{\left( \int_{M}S\right) ^2}
{{\rm Vol}(M^n)}+\int_{M} |\nabla A|^2  \geq\int_{M} |\nabla A|^2.
\end{aligned}
$$
Then it follows from Lemma  \ref{lem A similar Simons Inequality of minimal hypersurfaces in the spheres}, condition $(ii)$ and (\ref{equation Simons equation}) that
$$
\int_{M}S\left(S-n\right)
(S_{\max}-S_{\min}-\frac{n}{4})
=
(S_{\max}-S_{\min}-\frac{n}{4})
\int_{M} |\nabla A|^2
\geq 0.
$$
Therefore by condition $(iii)$, $|\nabla A|\equiv 0$ and $S\equiv n$ or $S\equiv 0$, i.e., $M^n$ is a  minimal Clifford torus or an equator.
The proof is similar
if Yau  Conjecture \ref{conj  Yau 1982} (i)  is true.
\end{proof}

\begin{lem}  \label{lem  the lower bonud of psiS}
Let $ M^n$ be a closed   immersed, non-totally geodesic,   minimal hypersurface  in  $\mathbb{S}^{n+1}$, then for all $a\in \mathbb{S}^{n+1}$, we have
\begin{equation*}
\int_{M}S\psi_{a}^2
\geq\lambda_1(M)
\frac{\left( \int_{M}S\right)^2 \int_{M}\psi_{a}^2}{ {\rm Vol}(M^n)\int_{M}S^2}.
\end{equation*}
\end{lem}
\begin{proof}
For all $a\in \mathbb{S}^{n+1}$, by Proposition \ref{prop funda} we have
$$
\Delta \psi_{a}=-S\psi_{a},\ \ \int_{M}S\psi_{a}=0.
$$
Hence, for all constant $K\in \mathbb{R}$, one has
\begin{equation*}
\left( \int_{M}\psi_{a}\right) ^2=\left(\int_{M}\left( 1-KS\right)\psi_{a} \right) ^2\leq\int_{M}\left( 1-KS\right)^2
\int_{M}\psi_{a}^2.
\end{equation*}
Set $$K=\frac{\int_{M}S }{\int_{M}S^2}, $$
then by similar argument as in Lemma \ref{lem A similar Simons Inequality of minimal hypersurfaces in the spheres}, we deduce
$$
\begin{aligned}
\int_{M}S\psi_a^2&=\int_{M}-\psi_{a}\Delta\psi_a=\int_{M}|\nabla \psi_a|^2
\geq\lambda_1(M)\left(
 \int_{M}\psi_a^2-
\frac{\left( \int_{M}\psi_a\right) ^2}
{{\rm Vol}(M^n)}\right) \\
&\geq \lambda_1(M)\frac{ \int_{M}\left( 2KS-K^2S^2\right)}{{\rm Vol}(M^n)}
\int_{M}\psi_{a}^2
=\lambda_1(M)
\frac{\left( \int_{M}S\right)^2 \int_{M}\psi_{a}^2}{ {\rm Vol}(M^n)\int_{M}S^2}.
\end{aligned}
$$
\end{proof}

\begin{rem}
If Yau Conjecture $\ref{conj  Yau 1982}$ $(i)$ is true, i.e.,
$\lambda_1(M)=n$ for all embedded minimal hypersurface of
$\mathbb{S}^{n+1}$, we have
$$
\int_{M}S\psi_{a}^2
\geq
n
\frac{\left( \int_{M}S\right)^2 \int_{M}\psi_{a}^2}{ {\rm Vol}(M^n)\int_{M}S^2},
$$
for all $a\in \mathbb{S}^{n+1}$.
Hence, summing up this inequality over an orthonormal basis $\{a_j\}_{j=1}^{n+2}$ for $a=a_j$ and noting that $\sum_{j=1}^{n+2}\psi_{a_j}^2=1$, we obtain $\int_{M}S^2\geq n \int_{M}S$. This gives another  proof of the Simons inequality for embedded hypersurfaces.
\end{rem}

\begin{lem}  \label{lem  the upper  bonud of psiS}
Let $ M^n$ be a closed  immersed  minimal hypersurface  in  $\mathbb{S}^{n+1}$, then for all $a\in \mathbb{S}^{n+1}$, we have
\begin{equation*}
\int_{M}S\psi_{a}^2
\leq
\frac{n-1}{2n-1}
\int_{M}S.
\end{equation*}
\end{lem}
\begin{proof}
Let $\{\lambda_i\}_{i=1}^n$  be the eigenvalues of $A$ with $\lambda_1^2\geq\lambda_2^2\geq\cdots\geq\lambda_n^2$. Then we have $$\sum_{i=1}^{n}\lambda_i=0,\quad \sum_{i=1}^{n}\lambda_i^2=\|A\|^2=S.$$
Thus
$$
\begin{aligned}
0=\Big(\sum_{i=1}^{n}\lambda_i\Big)^2
&=\lambda_1^2+
2\lambda_1\sum_{i=2}^{n}\lambda_i+
\Big(\sum_{i=2}^{n}\lambda_i\Big)^2=-\lambda_1^2+\Big(\sum_{i=2}^{n}\lambda_i\Big)^2\\
&\leq-\lambda_1^2+(n-1)\sum_{i=2}^{n}\lambda_i^2=(n-1)S-n\lambda_1^2.
\end{aligned}
$$
Hence
\begin{equation*}
\lambda_1^2\leq\frac{n-1}{n}S,
\end{equation*}
where the equality holds if and only if $ \lambda_1=(1-n)\lambda_2$ and $\lambda_2=\lambda_3=\dots=\lambda_n$. It follows from Proposition \ref{prop funda} that
$$
\begin{aligned}
\int_{M}S\psi_a^2=
&
\int_{M}-\psi_{a}\Delta\psi_a=\int_{M}|\nabla\psi_a|^2=
\int_{M}\left\langle Aa^{\rm{T} },  Aa^{\rm{T} }\right\rangle\\
&\leq\int_{M}\lambda_1^2\left\langle a^{\rm{T} },  a^{\rm{T} }\right\rangle
 =\int_{M}\lambda_1^2\left( 1-\psi_{a}^2-\varphi_a^2\right)\\
&\leq \frac{n-1}{n}\int_{M}S \left( 1-\psi_{a}^2-\varphi_a^2\right).
\end{aligned}
$$
Then
\begin{equation*}
   \int_{M}S\psi_{a}^2
   \leq
   \frac{n-1}{2n-1}
   \int_{M}S\left( 1-\varphi_a^2\right)
   \leq
   \frac{n-1}{2n-1}
   \int_{M}S.
   \qedhere
\end{equation*}
\end{proof}

\begin{lem}  \label{lem equator vanish}
Let $M^n$ be a closed  immersed    hypersurface  in  $\mathbb{S}^{n+1}$, then there exists an equator $\mathbb{S}^{n}\subset \mathbb{S}^{n+1}$ such that for all $a\in \mathbb{S}^{n}$,
$$
\int_{M}\psi_{a}=0.
$$
\end{lem}

\begin{proof}
Observe that the function $\mathcal{I}(a)$ on $a\in\mathbb{R}^{n+2}$ defined by $\mathcal{I}(a)=\int_{M}\psi_{a}$ is a linear function. Therefore the kernel of $\mathcal{I}(a)$ is a linear subspace of $\mathbb{R}^{n+2}$ of dimension at least $n+1$, whose intersection with $\mathbb{S}^{n+1}$ is the required equator.
\end{proof}
\begin{rem}\label{our conj}
For a closed immersed (or embedded), non-totally geodesic,  minimal hypersurface $M^n$  in  $\mathbb{S}^{n+1}$, we conjecture that $\int_{M}\psi_{a}=0$ for all $a\in\mathbb{S}^{n+1}$. For example,  if $M^n$ is invariant under the antipodal map and  the unit normal vector field along $M^n$ is odd, i.e., $\nu (-x)=-\nu (x)$, then the conjecture holds. 
 This would improve the inequality of Theorem $\ref{Corollary  gap  results of Perdomo Conjecture in the introduction}$ into $$\int_M S\geq \lambda_1(M){\rm Vol}(M^n).$$
Furthermore, it would prove the inequality part of Perdomo Conjecture $\ref{Conjecture Perdomo 2004}$ if Yau Conjecture $\ref{conj  Yau 1982}$ $(i)$ is true.
\end{rem}

\section{Proof of the theorems}
\begin{proof}[\textbf{Proof of Theorem $\mathbf{\ref{thm  gap  results of Perdomo Conjecture in the introduction}}$}]
 Let $M^n$ be a  closed  embedded hypersurface in $\mathbb{S}^{n+1}$. Denote the components of $\mathbb{S}^{n+1}\backslash M^n$  by $U_1$ and $U_2$, then we have $\mathbb{S}^{n+1}=U_1\bigcup_{M^n} U_2$. Obviously, we can extend the height function $\varphi_a(x)=\langle x,a\rangle$ on $M^n$ to $\mathbb{S}^{n+1}$. It is easy to prove that on $x\in\mathbb{S}^{n+1}$, $$\overline{\nabla}\varphi_a=a^{\rm\widetilde{T}},\ \ \overline{\Delta}\varphi_a=-(n+1)\varphi_a,$$
where  $a^{\rm\widetilde{T}}=a-\varphi_a x\in \Gamma(T\mathbb{S}^{n+1})$ denotes the tangent component of $a$ at $x\in\mathbb{S}^{n+1}$.  Then for all $a\in \mathbb{S}^{n+1}$ and $i\in \{1,2\}$, by the divergence theorem we have
$$
\left| \int_{M}\psi_{a}\right| =
\left| \int_{U_i}{\rm div}\left( a^{\rm \widetilde{T}}\right) \right|=
\left| \int_{U_i}\overline{\Delta}\varphi_{a}\right|=(n+1)
\left| \int_{U_i}\varphi_{a}\right|.
$$
By
\begin{equation*}\label{equation integral of varpsi}
\left| \int_{U_i}\varphi_{a}\right|\leq \int_{U_i}\left| \varphi_{a}\right|,
\end{equation*}
and
\begin{equation*}\label{equation psi equal Sn integral}
\int_{U_1}\left| \varphi_{a}\right|+
\int_{U_2}\left| \varphi_{a}\right|=
\int_{\mathbb{S}^{n+1}}\left| \varphi_{a}\right|
=2{\rm Vol}\left(\mathbb{B}^{n+1} \right)=
\frac{2}{n+1}{\rm Vol}\left(\mathbb{S}^{n} \right),
\end{equation*}
we can choose some $i_0\in \{1,2\}$ for every fixed $a\in \mathbb{S}^{n+1}$,  such that
\begin{equation}\label{equation psi integral}
\left| \int_{M}\psi_{a}\right| =
(n+1)\left| \int_{U_{i_0}}\varphi_{a}\right|\leq (n+1)
 \int_{U_{i_0}}\left| \varphi_{a}\right| \leq
{\rm Vol}\left(\mathbb{S}^{n} \right).
\end{equation}
By Lemma \ref{lem equator vanish}, we can choose an orthonormal basis $\{a_j\}_{j=1}^{n+2}$ of $\mathbb{R}^{n+2}$ such that $a_i$ $(i=1,\cdots,n+1)$ lie in the kernel of $\mathcal{I}(a)$, i.e.,
\begin{equation}\label{basis-ai}
\mathcal{I}(a_i)=\int_{M}\psi_{a_i}=0,\ \ i=1,\cdots,n+1.
\end{equation}
Let $\{e_k\}_{k=1}^n$ be a local orthonormal frame of $M^n$ such that $Ae_k=\lambda_k e_k$. Then the tangent component $a^{\rm T}$ along $M^n$ and $Aa^{\rm T}$ can be expressed by $$a^{\rm T}=\sum_{k=1}^n\langle a, e_k\rangle e_k, \ \ Aa^{\rm T}=\sum_{k=1}^n\langle a, e_k\rangle \lambda_k e_k.$$
This implies 
$$\sum_{j=1}^{n+2}|A{a_j}^{\rm T}|^2=\sum_{j=1}^{n+2}\sum_{k=1}^n\langle a_j, e_k\rangle^2 \lambda_k^2=\sum_{k=1}^n \lambda_k^2=|A|^2=S.$$
Then, by
$\sum_{j=1}^{n+2}
\psi_{a_j}^2=1$,
 $(\ref{equation psi integral}, \ref{basis-ai})$ and Proposition $\ref{prop funda}$, we deduce
$$
\begin{aligned}
\int_{M}S
&= 
\int_{M}\sum_{j=1}^{n+2}|A{a_j}^{\rm T}|^2
=\sum_{j=1}^{n+2}\int_{M}|\nabla\psi_{a_j}|^2\\
&\geq\lambda_1(M)\sum_{j=1}^{n+2}
\left(\int_{M}\psi_{a_j}^2-\frac{\left(\int_{M}\psi_{a_j} \right)^2 }{{\rm Vol}(M^n)} \right) \\
&= \lambda_1(M)\left( {\rm Vol}(M^n) -
\frac{\left(\int_{M}\psi_{a_{n+2}} \right)^2}{{\rm Vol}(M^n)}
\right) \\
&\geq \lambda_1(M)
\frac{{\rm Vol}^2(M^n)-{\rm Vol}^2(\mathbb{S}^{n})}{{\rm Vol}(M^n)}.
\end{aligned}
$$
This proves the inequality of the theorem. If the equality holds, the equal signs of  (\ref{equation psi integral}) also hold. Then there exists some $a=a_{n+2}\in \mathbb{S}^{n+1}$ such that
$$
\varphi_{a}\geq 0 \ \  \left(or\ \varphi_{a}\leq 0 \right)\ \   on\ U_{i_0},
$$
and thus $U_{i_0}$ lies in the hemisphere $\mathbb{S}_+^{n+1}:=\{x\in\mathbb{S}^{n+1}\mid \varphi_{a}(x)\geq 0 \}$, moreover,
\begin{equation*}
 \int_{U_{i_0}} \varphi_{a}
 =\frac{1}{n+1}
{\rm Vol}\left(\mathbb{S}^{n} \right).
\end{equation*}
On the other hand, we have  $$\int_{\mathbb{S}_+^{n+1}} \varphi_{a}=
\frac{1}{n+1} {\rm Vol}\left(\mathbb{S}^{n} \right).$$
Therefore we have $\overline{U_{i_0}}=\mathbb{S}_+^{n+1}$ and thus $M^n=\mathbb{S}^{n}$ is totally geodesic. 
\end{proof}


\begin{proof}[\textbf{Proof of Theorem $\mathbf{\ref{thm  some results of Perdomo Conjecture in the introduction}}$}]
\textbf{Case $({\rm i})$.}
As in the proof of Theorem \ref{thm  gap  results of Perdomo Conjecture in the introduction}, by Lemma \ref{lem equator vanish}, we can choose an orthonormal basis $\{a_j\}_{j=1}^{n+2}$ of $\mathbb{R}^{n+2}$ such that
\begin{equation*}
\int_{M}\psi_{a_i}=0,\ \ i=1,\cdots,n+1.
\end{equation*}
Hence by Proposition $\ref{prop funda}$,
\begin{equation}\label{Spsi lambda2 }
\int_{M}S\psi_{a_i}^2=\int_{M}|\nabla\psi_{a_i}|^2\geq
\lambda_1(M)\int_{M}\psi_{a_i}^2
\ \ i=1,\cdots,n+1.
\end{equation}
Due to Lemma \ref{lem  the lower bonud of psiS} and (\ref{Spsi lambda2 }), we have
$$
\sum_{i=1}^{n+1}\int_{M}S\psi_{a_i}^2+\frac{ {\rm Vol}(M^n)\int_{M}S^2}{\left( \int_{M}S\right)^2 }
\int_{M}S\psi_{a_{n+2}}^2
\geq\lambda_1(M)\sum_{i=1}^{n+2}
\int_{M}\psi_{a_i}^2.
$$
 Since $\sum_{i=1}^{n+2}\psi_{a_i}^2=1$, we have
\begin{equation}\label{S lambda var}
\int_{M} S +
\frac{ {\rm Vol}(M^n)\int_{M}S^2-\left( \int_{M}S\right)^2}
{\left( \int_{M}S\right)^2 }
\int_{M}S\psi_{a_{n+2}}^2\geq
 \lambda_1(M){\rm Vol}(M^n).
\end{equation}
By Lemma \ref{lem A similar Simons Inequality of minimal hypersurfaces in the spheres}, Lemma \ref{lem  the upper  bonud of psiS}, (\ref{equation cauchy inequ}) and  (\ref{S lambda var}),
 we obtain
$$
\int_{M}S +
\frac{ \int_{M}S(S-n)(S-S_{\min})}
{\int_{M}S}
\frac{2(n-1){\rm Vol}(M^n)}{\lambda_1(M)(2n-1)}\geq \lambda_1(M){\rm Vol}(M^n).
$$
By Theorem \ref{thm J. Simons} and $\lambda_1(M)\leq n$ for minimal hypersurfaces, without loss of generality, we can suppose $S_{\min}\leq n\leq S_{\max}$.  It follows that $$(S-n)(S-S_{\min})\leq(S_{\max}-n)(S_{\max}-S_{\min}).$$ 
Thus we obtain the required inequality
$$\begin{aligned}
 \int_{M}S
  &\geq
 \left( \lambda_1(M)-\frac{2(n-1)}{\lambda_1(M)(2n-1)}\left(S_{\max}-n \right) \left( S_{\max}-S_{\min}\right) \right){\rm Vol}(M^n). 
 \end{aligned}
$$

\textbf{Case $({\rm ii})$.}
By Lemma  \ref{lem  the upper  bonud of psiS},  (\ref{equation cauchy inequ}) and  (\ref{S lambda var}), we have
$$
\int_{M} S +
\frac{n-1}{2n-1}
\frac{ {\rm Vol}(M^n)\int_{M}S^2-\left( \int_{M}S\right)^2}
{\int_{M}S }\geq
 \lambda_1(M) {\rm Vol}(M^n). 
$$
Thus
$$
\int_{M}S^2
\geq
\frac{n}{n-1}\int_{M}S
\left( \frac{2n-1}{n}\lambda_1(M)-\frac{\int_{M}S}{{\rm{Vol}}(M^n)  }\right),
$$
which implies 
$$
n\int_{M}S+(n-1)S_{\max}{\rm{Vol}}(M^n)\geq (2n-1)\lambda_1(M){\rm{Vol}}(M^n),
$$
if $S\not\equiv 0$.
\end{proof}

\begin{proof}[\textbf{Proof of Corollary $\mathbf{\ref{Corollary the week result of Perdomo Conjecture}}$}]
If $S\not\equiv 0$, then by  Theorem \ref{thm  some results of Perdomo Conjecture in the introduction} we have
$$ \int_{M}\left( S+\frac{2(n-1)}{\lambda_1(M)(2n-1)}\left(S_{\max}-n \right) \left( S_{\max}-S_{\min}\right) \right)\geq
 \lambda_1(M){\rm{Vol}}(M^n).$$
Substituting condition $(ii)$ $S_{\max}<S_{\min}+n$ into the inequality above, we get
$$\int_{M}S \left(1+\frac{2(n-1)n}{\lambda_1(M)(2n-1)}\right)\geq\int_{M} \left(S+\frac{2(n-1)n}{\lambda_1(M)(2n-1)}S_{\min}\right)>
 \lambda_1(M){\rm{Vol}}(M^n).$$
 Therefore
 $$\int_{M}S> \frac{(2n-1)\lambda_1^2(M)}{(2n-1)\lambda_1(M)+2n(n-1)}{\rm{Vol}}(M^n),$$
 contradicting to condition $(i)$.
\end{proof}

\begin{acknow}
The authors sincerely thank Professors Aldir Brasil and Oscar M. Perdomo for their interests and  suggestions. They also thank the referee for their helpful comments.
\end{acknow}

\end{document}